\documentclass[11pt]{amsart}

\usepackage[dvips]{graphicx}
\usepackage{calc}
\usepackage{color}
\usepackage{amsmath}
\usepackage{amssymb}
\usepackage{amscd}
\usepackage{amsthm}
\usepackage{amsbsy}
\usepackage{delarray}
\usepackage{enumerate}
\usepackage[T1]{fontenc}
\usepackage{inputenc}
\usepackage{enumerate}
\usepackage{hyperref}

\begin{document}



\setlength{\parindent}{5mm}
\renewcommand{\leq}{\leqslant}
\renewcommand{\geq}{\geqslant}
\newcommand{\N}{\mathbb{N}}
\newcommand{\sph}{\mathbb{S}}
\newcommand{\Z}{\mathbb{Z}}
\newcommand{\R}{\mathbb{R}}
\newcommand{\C}{\mathbb{C}}
\newcommand{\F}{\mathbb{F}}
\newcommand{\g}{\mathfrak{g}}
\newcommand{\h}{\mathfrak{h}}
\newcommand{\K}{\mathbb{K}}
\newcommand{\RN}{\mathbb{R}^{2n}}
\newcommand{\ci}{c^{\infty}}
\newcommand{\derive}[2]{\frac{\partial{#1}}{\partial{#2}}}
\renewcommand{\S}{\mathbb{S}}
\renewcommand{\H}{\mathbb{H}}
\newcommand{\eps}{\varepsilon}
\newcommand{\chz}{c_{\mathrm{HZ}}}

\theoremstyle{plain}
\newtheorem{theo}{Theorem}
\newtheorem{prop}[theo]{Proposition}
\newtheorem{lemma}[theo]{Lemma}
\newtheorem{definition}[theo]{Definition}
\newtheorem*{notation*}{Notation}
\newtheorem*{notations*}{Notations}
\newtheorem{corol}[theo]{Corollary}
\newtheorem{conj}[theo]{Conjecture}
\newtheorem{question}[theo]{Question}
\newtheorem*{question*}{Question}

\newenvironment{demo}[1][]{\addvspace{8mm} \emph{Proof #1.
    ---~~}}{~~~$\Box$\bigskip}

\newlength{\espaceavantspecialthm}
\newlength{\espaceapresspecialthm}
\setlength{\espaceavantspecialthm}{\topsep} \setlength{\espaceapresspecialthm}{\topsep}

\newenvironment{example}[1][]{\refstepcounter{theo} 
\vskip \espaceavantspecialthm \noindent \textsc{Example~\thetheo
#1.} }%
{\vskip \espaceapresspecialthm}

\newenvironment{remark}[1][]{\refstepcounter{theo} 
\vskip \espaceavantspecialthm \noindent \textsc{Remark~\thetheo
#1.} }%
{\vskip \espaceapresspecialthm}

\def\Homeo{\mathrm{Homeo}}
\def\Hameo{\mathrm{Hameo}}
\def\Diffeo{\mathrm{Diffeo}}
\def\Symp{\mathrm{Symp}}
\def\Id{\mathrm{Id}}
\newcommand{\norm}[1]{||#1||}
\def\Ham{\mathrm{Ham}}
\def\Hamtilde{\widetilde{\mathrm{Ham}}}
\def\Crit{\mathrm{Crit}}
\def\Spec{\mathrm{Spec}}
\def\osc{\mathrm{osc}}
\def\Cal{\mathrm{Cal}}

\title[Energy-Capacity inequalities and continuous Hamiltonians]{New Energy-Capacity-type inequalities and uniqueness of continuous Hamiltonians} 
\author{Vincent Humilière, Rémi Leclercq, Sobhan Seyfaddini}
\date{\today}

\address{VH: Institut de Math\'ematiques de Jussieu, Universit\'e Pierre et Marie Curie, 4 place Jussieu, 75005 Paris, France}
\email{vincent.humiliere@imj-prg.fr}

\address{RL: Universit\'e Paris-Sud, D\'epartement de Math\'ematiques, Bat. 425, 91405 Orsay Cedex, France}
\email{remi.leclercq@math.u-psud.fr}

\address{SS: University of California Berkeley, Berkeley, CA 94720, USA}
\email{sobhan@math.berkeley.edu}


\maketitle

\begin{abstract} 
We prove a new variant of the energy-capacity inequality for closed rational symplectic manifolds (as well as certain open manifolds such as $\R^{2n}$, cotangent bundle of closed manifolds...) and we derive some consequences to $C^0$--symplectic topology. Namely, we prove that a continuous function which is a uniform limit of smooth normalized Hamiltonians whose flows converge to the identity for the spectral (or Hofer's) distance must vanish. This gives a new proof of uniqueness of continuous generating Hamiltonian for hameomorphisms. This also allows us to improve a result by Cardin and Viterbo on the $C^0$--rigidity of the Poisson bracket. 
\end{abstract}

\section{Introduction and results}

Let $(M, \omega)$ denote a closed and connected symplectic manifold.  It is said to be rational if $\omega(\pi_2(M)) = \Omega \mathbb{Z}$ for a non-negative $\Omega \in \mathbb{R}$.  A rational symplectic manifold is called monotone if there exists $\lambda \in \mathbb{R}$ such that $[\omega] = \lambda c_1$ on $\pi_2(M)$, where $c_1$ denotes the first Chern class of $(M,\omega)$.  We say that $M$ is positively monotone if $\lambda \geq 0$ and negatively monotone if $\lambda < 0$.

Recall that, because $\omega$ is non-degenerate, a smooth Hamiltonian, that is, a smooth map $H:\S^1\times M\to \R$, generates a family of Hamiltonian vector fields defined by $dH_t=\omega(X_H^t,\cdot\,)$ and which in turn generates a 1--parameter family of diffeomorphisms $\phi_H^t$ such that $\phi_H^0$ is the identity and $\partial_t \phi_H^t = X_H^t(\phi_H^t)$. 

The time--1 diffeomorphisms obtained as the end of a Hamiltonian flow form a group called the Hamiltonian diffeomorphism group and usually denoted $\Ham(M, \omega)$. Its universal cover, $\Hamtilde(M, \omega)$, is naturally isomorphic to the set of equivalence classes of normalized Hamiltonians. Recall that (on compact manifolds), a Hamiltonian is said to be normalized if for all $t$, $\int_M H_t\, \omega^n=0$ and that two normalized Hamiltonians $H$ and $K$ are equivalent if there exists a homotopy running from $H$ to $K$, consisting of normalized Hamiltonians whose flows have fixed ends, namely $\Id$ and $\phi:=\phi_H^1=\phi_K^1$.

The universal cover $\Hamtilde(M, \omega)$ admits two natural ``(pseudo-)norms''. The first one was introduced by Hofer in \cite{hofer} (and is now called Hofer's norm). It is defined by
$$\|\tilde{\phi}\|= \inf_K\int_0^1\left(\max_{x\in M} K(t,x)-\min_{x\in M} K(t,x)\right)dt$$
where the infimum is taken over all Hamiltonians $K$ whose flow is a representative of the homotopy class $\tilde{\phi}$.

The second one arises as a consequence of the theory of spectral invariants. One can associate to every smooth Hamiltonian a real number called the spectral invariant of $H$; it is usually denoted by $c(1,H)$.  This is, roughly speaking, the action level at which the neutral element $1\in QH^*(M)$ appears in the Floer homology of $H$. These invariants were introduced by Viterbo, Schwarz and Oh (See \cite{viterbo1}, \cite{schwarz}, the lecture notes \cite{Oh} and references therein). They have been extensively studied and have had many interesting applications to symplectic topology. For example, they were used by Entov and Polterovich in their construction of Calabi quasimorphisms \cite{entov-polterovich1}, and by Ginzburg in his proof of the Conley conjecture \cite{Ginzburg}.

Note that, even though the unit of the quantum cohomology ring is not necessarily the only class to which one can associate such invariants, it is the only one used in this article and thus $c(1,H)$ will be denoted $c(H)$.

Spectral invariants lead to a ``spectral pseudo-norm'' which is defined for an element $\tilde{\phi}\in \widetilde{\Ham}(M,\omega)$ generated by a Hamiltonian $H$ as
$$\gamma(\tilde{\phi})=c(H)+c(\bar{H})$$
($\bar{H}$ is explicitly defined in Section \ref{section spectral invariants}, it generates the Hamiltonian isotopy $(\phi_H^t)^{-1}$). One quite remarkable fact is that the spectral pseudo-norm is bounded from above by Hofer's norm (see Section \ref{section spectral invariants}).

In this article, we are interested in limits of Hamiltonian flows for these (pseudo-)norms; this is a central theme of what is now called ``$C^0$--symplectic topology''. This terminology refers to a family of problems in symplectic topology that tries to define and study continuous analogs of the classical smooth objects of the symplectic world. Such definitions are often made possible by symplectic rigidity results. As an example, the famous Gromov--Eliashberg Theorem (the group of symplectic diffeomorphisms is $C^0$--closed in the full group of diffeomorphisms) allows to define a symplectic homeomorphism as a homeomorphism which is a $C^0$--limit of symplectic diffeomorphisms. 

One important motivation for $C^0$--symplectic topology is to try to define continuous Hamiltonian dynamics. 
As an example, this was the purpose of the definition by Oh and Müller \cite{muller-oh} of the notion of a ``continuous Hamiltonian isotopy'' (we will contract this terminology to the shorter ``hameotopy''), whose definition we now recall.  Equip $M$ with a distance $d$ induced by any Riemannian metric.  We define the $C^0$--distance between two homeomorphisms $\phi, \psi$ by  $\displaystyle d_{C^0}(\phi, \psi) := \max_{x} d(\phi(x), \psi(x)).$ For two paths of homeomorphisms $\phi^t, \psi^t$ ($t\in [0,1]$) define  $\displaystyle d_{C^0}(\phi^t, \psi^t):= \max_{t,x} d(\phi^t(x), \psi^t(x)).$ In Remark \ref{rem: C0_dist}, we briefly discuss an important property of this metric.

 A path of homeomorphisms $h^t$ is a \emph{hameotopy} if there exists a sequence of smooth Hamiltonian functions $\{H_k\}$ such that
 \begin{itemize}
 \item $d_{C^0}(\phi_{H_k}^t, h^t) \to 0,$ 
 \item the Hamiltonian functions $H_k$ converge uniformly to a continuous function $H:\S^1\times M\to\R$.
\end{itemize}

Analogously to the smooth case, the function $H$ is said to ``generate'' the isotopy $h^t$. A continuous function $H$ generates at most one hameotopy \cite{muller-oh}.
The set of all time-independent functions $H$ generating a hameotopy will be denoted by $C^0_{\Ham}$. As noticed in \cite{muller-oh}, every $C^{1,1}$ function belongs to $C^0_{\Ham}$. One important result of the theory is the uniqueness of the generating continuous Hamiltonian:
\begin{theo}[Viterbo \cite{viterbo2}, Buhovsky--Seyfaddini \cite{buhovsky-seyfaddini}]\label{theo unicite hameo} Let $\{H_k\}$, $\{H_k'\}$ be two sequences of normalized smooth Hamiltonians on a closed manifold $M$. Suppose that their flows $C^0$--converge to the same continuous isotopy and that $H_k-H_k'$ converges uniformly to some continuous function $H$. Then $H$ vanishes identically. 
In other words, given a hameotopy, the generating continuous Hamiltonian is unique.
\end{theo}

The first major theorem of this article is a result analogous to the above with the $C^0$--distance replaced by the spectral pseudo-distance $\gamma$.  Since, in this generality, $\gamma$ is not defined on the Hamiltonian diffeomorphism group itself but only on its universal cover, we need to replace isotopies by their lift to the universal cover. We will denote by $\{\tilde{\phi}_H^t\}$ (or just $\tilde{\phi}_H^t$) the unique lift of the isotopy $\{\phi_H^t\}$ to $\Hamtilde(M,\omega)$ whose starting point, $\tilde{\phi}_H^0$, is the identity element. Said differently, for fixed $t\in\R$, $\tilde{\phi}_H^t$ is the element of $\Hamtilde(M,\omega)$ represented by the path $[0,1]\to \Ham(M,\omega)$, $s\mapsto \phi_H^{st}$.
\begin{theo}\label{theo unicite complete} Let $(M, \omega)$ denote a rational symplectic manifold, let $U$ be a non-empty open subset of $M$, $I$ be a non-empty open interval in $\R$ and $\{H_k\}$, $\{H_k'\}$ be two sequences of smooth Hamiltonians such that
\begin{itemize}
 \item[(i)]  For any $t\in I$, $\gamma(\tilde{\phi}_{H_k}^t,\tilde{\phi}_{H_k'}^t)$ converges to zero,
 \item[(ii)] $H_k$ and $H_k'$ converge uniformly on $I\times U$ respectively to continuous functions $H$ and $H'$.
\end{itemize}
Then, $H-H'$ depends only on the time variable on $I\times U$. 
\end{theo}
 Note that, since the spectral pseudo-distance is bounded from above by Hofer's distance, this theorem also holds with $\gamma$ replaced by $\|\cdot\|$. Note also that if $U=M$ and if the sequences consist of normalized Hamiltonians, then $H=H'$. 

Finally, let us emphasize the fact that if spectral invariants descend from $\Hamtilde(M,\omega)$ to $\Ham(M,\omega)$, then $\gamma$ also descends (as a genuine norm) and Theorem \ref{theo unicite complete} holds if we replace (i) by the much weaker assumption:
\begin{itemize}
 \item[(i')]For any $t\in I$,  $\gamma(\phi_{H_k}^t,\phi_{H_k'}^t)$ converges to zero.
\end{itemize}
For example, this is true if we assume the additional (rather strong) assumption that $(M, \omega)$ is weakly exact (that is, $\omega(\pi_2(M))=0$). Another example comes from the third author's \cite{Seyfaddini12}. Assume that $(M, \omega)$ is negatively monotone and that there exists a non-empty open set $V$ such that for all $k$, $H_k$ and $H'_k$ lie in $C_c^\infty(\S^1\times (M\backslash V))$, then Theorem \ref{theo unicite complete} holds under (i') and (ii).

\subsection*{Applications}
Theorem \ref{theo unicite hameo}, which proves that hameotopies have unique normalized generating Hamiltonians, is one of the most foundational results in $C^0$ Hamiltonian dynamics; see \cite{buhovsky-seyfaddini, muller-oh,viterbo2} for some of the consequences of this theorem.  In Section \ref{section preuve unicite}, we will show that Theorem \ref{theo unicite complete} allows us to recover Theorem \ref{theo unicite hameo}; to the best of our knowledge, this is the first proof of Theorem \ref{theo unicite hameo} via Floer-theoretic methods.   

\medskip
In addition to the above, Theorem \ref{theo unicite complete} has other interesting consequences as well. In \cite{humiliere1} (see \cite{humiliere-these} for a better presentation though in French), the first author suggested another attempt of defining continuous Hamiltonian dynamics. The idea is to introduce the abstract completion of the group of Hamiltonian diffeomorphisms with respect to the spectral metric. The paper is written in $\R^{2n}$ but everything there can be done on general, symplectically aspherical, closed manifolds (where, as mentioned above, $\gamma$ descends to a non-degenerate norm on $\Ham(M,\omega)$). On the level of Hamiltonian functions, one can introduce a distance between two Hamiltonians by $$\gamma_u(H,K)=\sup_{t\in[0,1]}\gamma(\phi_H^{t},\phi_K^{t}),$$
and call a ``generalized Hamiltonian'' any element in the completion of the set of smooth Hamiltonians with respect to the distance $\gamma_u$. The canonical map $H\mapsto \phi_H^t$ naturally extends to the completions, and we can speak of the ``flow'' generated by a generalized Hamiltonian.
These completions have applications to the study of Hamilton--Jacobi equations (\cite{humiliere1, humiliere-these}). They are also needed for Viterbo's symplectic homogenization theory \cite{viterbo3}.

The main problem encountered with these completions is that their elements are a priori very abstract objects, that is, equivalence classes of Cauchy sequences for some abstract distance.
However, some elements can be represented by honest continuous functions: Indeed, the inequality $\gamma_u\leq \|\cdot\|_{C^0}$ induces a map $\iota$ from $C^0_c(\S^1\times M)$ to the set of generalized Hamiltonians. It follows that continuous Hamiltonians have a flow in the $\gamma$--completion of the Hamiltonian group.
Like in the case of hameotopies, it is natural to wonder whether the generating continuous Hamiltonian is unique. 
Theorem \ref{theo unicite complete} answers this question positively. 
It says in particular that the map $\iota$ is injective. In other words, the continuous function representing a given generalized Hamiltonian is unique.

 Note that since Theorem \ref{theo unicite complete} holds for any open set $U$, the uniqueness of the continuous generator is actually local. Therefore, the result can be applied to generalized Hamiltonians that can be represented by not everywhere continuous functions. Examples of such elements where provided in \cite{humiliere1}.

\medskip
Theorem \ref{theo unicite complete} also has consequences in terms of $C^0$--rigidity of the Poisson bracket.  Recall that the Poisson bracket of two differentiable functions $F$, $G$ on $M$, with Hamiltonian vector field $X_F$, $X_G$ is given by
$$\{F,G\}=\omega(X_F,X_G).$$
A function $F$ is called a first integral of $G\in C^0_{\Ham}$ if $F$ is constant along the flow of $G$. When $F$ and $G$ are smooth, $F$ is a first integral of $G$ if and only if $\{F,G\}=0$.

As one can see, the Poisson bracket is defined only in terms of the differentials of the involved functions. Nevertheless, it satisfies some rigidity with respect to the $C^0$--topology. This property was first discovered by Cardin and Viterbo \cite{cardin-viterbo}. Their theorem has opened an active domain of research and has been improved in several directions by many authors (see e.g. \cite{buhovsky, buhovsky-entov-polterovich, entov-polterovich2, entov-polterovich-rosen, humiliere2, zapolsky} for some of the strongest results). 

Here, we improve the result of Cardin and Viterbo in a new direction.
\begin{theo}\label{theo Poisson} Let $F_k$ and $G_k$ be two sequences of smooth functions on a closed, rational symplectic manifold $M$ such that:
\begin{itemize}
\item the sequence $F_k$ converges uniformly to some continuous function $F$,
 \item the sequence $G_k$ converges uniformly to some function $G\in C^0_{\Ham}$,
 \item the sequence of Poisson brackets $\{F_k,G_k\}$ converges uniformly to 0.
\end{itemize}
Then, $F$ is a first integral of $G$.
\end{theo}

In particular, the theorem holds when $F$ is $C^0$ and $G$ is $C^{1,1}$.
The result of Cardin and Viterbo was the same theorem but with both $F$ and $G$ of class $C^{1,1}$. In our case where $F$ is only $C^0$, the proof is made more difficult by the fact that $F$ does not have any flow in general. 

 After the first version of this paper was written, we were informed by Buhovsky that it is possible to prove Theorem \ref{theo Poisson} using the energy-capacity inequality. Furthermore, Buhovsky's method allows him to remove the rationality assumption in the statement of the theorem.

\medskip
Theorem \ref{theo Poisson} allows us to relate two notions of Poisson commutativity for continuous Hamiltonians. First recall the definition proposed by Cardin and Viterbo \cite{cardin-viterbo}: Two continuous functions $C^0$--\emph{commute} if they are uniform limits of functions whose Poisson bracket uniformly converges to 0. Another definition of commutativity for functions in $C^0_{\Ham}$ would simply be that their flows commute. Theorem \ref{theo Poisson} has the following immediate corollary. 
\begin{corol} If two functions in $C^0_{\Ham}$ commute in the sense of Cardin and Viterbo, then the hameotopies they generate commute.\end{corol}

\subsection*{Key technique involved in the proof of the main result}

In order to prove Theorem \ref{theo unicite complete}, we first establish a new variant of the energy-capacity inequality for closed monotone symplectic manifolds. 

We denote by $\chz$ the following version of the Hofer--Zehnder capacity: For an open set $U$, $$\chz(U)=\sup\{\max f\,|\,f\in C_c^\infty(U) \text{ slow and non-negative}\}.$$ 
Recall that $H$ is called \emph{slow} if its Hamiltonian flow $\{\phi_H^t\}_t$ has no non-trivial orbits of period at most $1$. As an example, it is well known that for a symplectic ball $B$ of radius $r$, $\chz(B)=\pi r^2$. 

Since these types of capacities are defined in a very different fashion than the action selector $c$ -- as well as other natural invariants like displacement energy --, comparison between them (energy-capacity-like inequalities) leads to interesting consequences (see e.g \cite[Theorem 1]{FGS} for such relations and further applications).  The key result toward our proof of Theorem \ref{theo unicite complete} is the following set of energy-capacity-like inequalities.
\begin{theo}\label{theo base}
  Let $(M, \omega)$ denote a monotone symplectic manifold.  Suppose that $U$ is an open subset of $M$ and $H$ is a smooth Hamiltonian such that $ \forall (t,x) \in [0,1] \times U$ we have $H(t,x) = C$.  Then, at least one of the following two possibilities holds:
  \begin{enumerate}
  \item $\gamma(\tilde{\phi}_H^1)=c(H) + c(\bar{H}) \geq \chz(U)$,
  \item $  |c( H) - C|  \leq  \chz(U)$ and $  |c( \bar{H}) + C|  \leq \chz(U)$.
  \end{enumerate}
\end{theo}

We will prove Theorem \ref{theo base} in Section \ref{section theo base}. It is evident from our proof that if $(M, \omega)$ is positively monotone and the second of the above possibilities holds, then the numbers ($c(H) - C)$ and $(c(\bar{H}) + C)$ are always non-negative.  

The above result, combined with the fact that $c$ and $\gamma$ are both bounded by the Hofer norm (see Section \ref{section spectral invariants}), has an immediate corollary.

\begin{corol}Let $(M, \omega)$ denote a monotone symplectic manifold, $U$ an open subset of $M$, and $H$ a normalized and smooth Hamiltonian such that for any $(t,x) \in [0,1] \times U$ we have $H(t,x) = C \geq 2 \chz(U)$.  Then $||\tilde{\phi}_H^1||\geq \chz(U)$.
\end{corol}

Our proof of Theorem \ref{theo base} relies on the discreteness of $\omega(\pi_2(M))$ and hence it does not extend 
to irrational manifolds.  However, when $(M, \omega)$ is rational, but not monotone, we can prove a weaker version of Theorem \ref{theo base} which is sufficient for the applications considered in this article. 

\begin{theo}\label{theo base rational}
  Let $(M, \omega)$ denote a rational symplectic manifold.  Suppose that $U$ is an open subset of $M$ and $H$ is a smooth Hamiltonian such that $ \forall (t,x) \in [0,1] \times U$ we have $H(t,x) = C $. Then, at least one of the following two possibilities holds:
  \begin{enumerate}
  \item $\gamma(\tilde{\phi}_H^1)=c(H) + c(\bar{H}) \geq \chz(U)$,
  \item there exist $k$ and $\bar{k} \in \mathbb{Z}$, depending on $H$, such that:\\ $0 \leq c( H) - C  - k \Omega \leq  \chz(U)$ and $0 \leq c( \bar{H}) + C + \bar{k} \Omega \leq \chz(U)$.
  \end{enumerate}  
Moreover, if $\chz(U) < \frac{1}{2} \Omega$ then we may choose $k=\bar{k}$.
\end{theo}
Theorem \ref{theo base rational} will be proven in Section \ref{section theo base}; note that it is trivially true if $\chz(U) \geq \Omega$.

Currently \cite{HLS}, we are in the process of proving energy-capacity-type inequalities, in the spirit of those appearing in this section, for Lagrangian spectral invariants as defined by Viterbo \cite{viterbo1} for cotangent bundles or by Leclercq \cite{Leclercq08} for weakly exact Lagrangians in compact manifolds.  Such inequalities could be potentially very helpful in obtaining new rigidity results for Lagrangian submanifolds.

\subsection*{Extension to non-closed manifolds}
In this article, we have written our results for closed manifolds only, but each of them can be adapted to non-closed manifolds as soon as spectral invariants are properly defined and satisfy the standard properties (see Proposition \ref{Properties of Spectral Invariants} below). Of course, in this case, we only consider compactly supported Hamiltonians. (Note that in non-compact manifolds the requirement for a Hamiltonian to have compact support is a natural -- and commonly used -- normalization condition.)

Frauenfelder and Schlenk \cite{frauenfelder-schlenk} defined the spectral invariant $c$ on any weakly exact convex at infinity symplectic manifold. This has been extended to more general convex at infinity symplectic manifolds by Lanzat \cite{lanzat}. In the special case of $\R^{2n}$ spectral invariants can be defined using generating functions instead of Floer homology following Viterbo \cite{viterbo1}. Our results also extend to this setting.

\subsection*{Acknowledgments} 
This paper solves a question which remained open in the first author's Ph.D. thesis defended a few years ago. The first author is grateful to his former supervisor Claude Viterbo for his support and to Patrick Bernard for interesting discussions around these questions, as he was finishing his Ph.D. He also thanks Felix Schlenk for answering his questions on the action selector on convex non-closed symplectic manifolds. The second and the third authors are grateful to the first author for generously inviting them to join the project at an intermediate stage.  The third author would like to thank Lev Buhovsky, Leonid Polterovich, and Alan Weinstein for helpful discussions.  Finally, we would like to thank an anonymous referee for carefully reading this paper and pointing out some inaccuracies.

This work is partially supported by the French \emph{Agence Nationale de la Recherche}, project ANR-11-JS01-010-01.

\section{A review of spectral invariants}\label{section spectral invariants}

In this section we briefly review the theory of spectral invariants on closed symplectic manifolds.  For further details we refer the interested reader to \cite{mcDuff-salamon, Oh, schwarz}.  
   
   Denote by $\Omega_0(M)$ the space of contractible loops in $M$ and let $\Gamma:= \frac{\pi_2(M)}{\ker(c_1) \cap \ker([\omega])}$.  It is the group of deck transformations of the Novikov covering of $\Omega_0(M)$, which is defined by the following expression:
 $$\tilde{\Omega}_0(M) = \frac{ \{ [z,u]: z \in \Omega_0(M) , u: D^2 \rightarrow M , u|_{\partial D^2} = z \}}{[z,u] = [z', u'] \text { if } z=z' \text{ and } \bar{u} \# u' = 0 \text{ in } \Gamma},$$
   where $\bar{u} \# u'$ denotes the sphere obtained by gluing $u$ and $u'$ along their common boundary with the orientation on $u$ reversed.  The disc $u$, appearing in the above definition, is referred to as a capping disc of $z$.  Recall that the action functional of a Hamiltonian $H$ is a map from $\tilde{\Omega}_0(M)$ to $\mathbb{R}$ defined by
   $$\mathcal{A}_H([z,u]) =  \int_{\S^1} H(t,z(t))dt \text{ }- \int_{D^2} u^*\omega.$$
   It is well known that the set of critical points of $\mathcal{A}_H$, denoted by $\Crit(\mathcal{A}_H)$, consists of equivalence classes of pairs,  $[z,u] \in \tilde{\Omega}_0(M)$, such that $z$ is a $1$--periodic orbit of the Hamiltonian flow $\phi^t_H$.  The set of critical values of $\mathcal{A}_H$ is called the action spectrum of $H$ and is denoted by $\Spec(H)$; it has Lebesgue measure zero. When $H$ is non-degenerate, the set $\Crit(\mathcal{A}_H)$ can be indexed by the well known Conley--Zehnder index, $\mu_\mathrm{CZ}: \Crit(\mathcal{A}_H) \rightarrow \mathbb{Z}$, for every $\Sigma \in \Gamma$, the Conley--Zehnder index satisfies
   \begin{equation}\label{CZ-index identity}
   \mu_\mathrm{CZ}([z,u\#\Sigma]) = \mu_\mathrm{CZ}([z,u]) - 2 c_1(\Sigma).
   \end{equation}
   Several conventions are used for defining this index.  We fix our convention in the following fashion: suppose that $g$ is a $C^2$--small Morse function.  For every critical point $p$ of $g$, we require that 
$$i_{\text{Morse}}(p) = \mu_\mathrm{CZ}([p, u_p]),$$
where $i_{\text{Morse}}(p)$ is the Morse index of $p$ and $u_p$ is the trivial capping disc. 
Notice that the set of equivalence classes of pairs $[p,u]$ consists of equivalence classes $[p,\Sigma]$, with $\Sigma$ in $\pi_2(M)$ (and $[p,\Sigma] = [p,\Sigma']$ if $\omega(\Sigma)=\omega(\Sigma')$) and that with our convention $\mu_\mathrm{CZ}([p,\Sigma]) = i_{\text{Morse}}(p) -2 c_1(\Sigma)$. This observation will be useful in the proofs of Theorems \ref{theo base} and \ref{theo base rational}.

Spectral invariants, or action selectors, are defined via Hamiltonian Floer theory.  The procedure consists of filtering Floer homology by the values of the action functional and then associating to quantum cohomology classes (seen as Floer homology classes via the so-called PSS homomorphism \cite{PSS}) the minimal action level at which they appear in the filtration. As mentioned in the introduction, the specific spectral invariant used in this article, denoted by $c(H)$ for $H \in C^{\infty}(\S^1 \times M)$, is the one associated to the neutral element $1\in QH^*(M)$.  We will now list, without proof, the basic properties of this spectral invariant.   Recall that the composition of two Hamiltonian flows, $\phi^t_H \circ \phi^t_G$, and the inverse of a flow, $(\phi^t_H)^{-1},$ are Hamiltonian flows generated by $H\#G(t,x) = H(t,x) + G(t, (\phi^t_H)^{-1}(x))$ and $\bar{H}(t,x) = -H(t, \phi^t_H(x))$, respectively.
\begin{prop} \label{Properties of Spectral Invariants}(\cite{Oh2, Oh, schwarz, usher})\\
    The spectral invariant $c: C^{\infty}(\S^1 \times M)  \rightarrow \mathbb{R}$ has the following properties:
    \begin{enumerate}   
    \item  (Shift) If $r : \S^1 \rightarrow \mathbb{R}$ is smooth then $c(H+r) = c(H) + \int_{\S^1}{r(t) dt}.$
    \item (Triangle Inequality) $c(H\#G) \leq c(H) + c(G)$.
    \item (Continuity)  $|c(H) - c(G)| \leq \int_{\S^1} \max_{x\in M} |H_t -G_t|  dt.$
    \item (Spectrality) If $(M, \omega)$ is rational, then there exists $[z,u] \in \Crit(\mathcal{A}_H)$ such that $c(H) = \mathcal{A}_H([z,u])$, i.e. $c(H) \in \Spec(H)$.  Furthermore, if $H$ is non-degenerate then $\mu_\mathrm{CZ}([z,u]) = 2n$.  
    \item (Homotopy Invariance) Suppose that $H$ and $G$ are normalized and generate the same element of $\widetilde{Ham}(M)$.  Then, $c(H) = c(G)$.  
    \end{enumerate}   
   \end{prop}
The spectral pseudo-norm $\gamma$ is defined on $\widetilde{\Ham}(M,\omega)$ by the expression 
$$ \gamma(\phi^t_H) = c(H) + c( \bar{H}).$$
It induces a pseudo-distance (also denoted $\gamma$) defined by
$$ \gamma(\phi^t_H, \phi^t_K) = \gamma((\phi^t_K)^{-1}\circ\phi^t_H) = c(\bar{K} \# H) + c( \bar{H} \# K).$$ 
Note that $c$ and $\gamma$ are both bounded by the Hofer distance $\|\cdot\|$. This easily follows from a slightly different version of Property \textit{(3)}:
$$\int_{\S^1} \min_{x\in M} (H_t -G_t)  dt \leq c(H) - c(G) \leq \int_{\S^1} \max_{x\in M} (H_t -G_t)  dt.$$

It is well-known \cite{schwarz} that if $\omega|_{\pi_2(M)} = 0$, then $\gamma$ descends to a genuine distance on $\Ham(M, \omega)$.

   Finally, we end this section with the following lemma which will be used in the proof of Theorem \ref{theo base}.
  \begin{lemma} \label{CZ index lemma}
    Suppose that $H$ is a not necessarily non-degenerate Hamiltonian on a symplectic manifold $(M, \omega)$. Let $A = \{ \gamma: \exists u \mbox{ s.t. } c(H) = \mathcal{A}_H([\gamma,u]) \}$.  If all the orbits $\gamma \in A$ are non-degenerate, then there exists a capped orbit $[\gamma, u]$ such that $c(H) = \mathcal{A}_H([\gamma,u])$ and $\mu_{CZ}([\gamma, u]) =  2n$.
  \end{lemma} 
  \begin{proof}
For each $\gamma\in A$, let $U_{\gamma}$ denote a neighborhood of $\gamma$ which contains no other periodic orbits of $H$; such neighborhoods exist because the orbits contained in $A$ are all isolated.  Pick a sequence of non-degenerate Hamiltonians, $\{H_i\}_i$, $C^3$--approximating $H$ such that for every $i$ and every $\gamma\in A$, $H_i|_{U_{\gamma}} = H|_{U_{\gamma}}$.  For each $i$, let $[\gamma_i, u_i]$ denote a capped orbit of  $H_i$ such that $c(H_i) = \mathcal{A}_{H_i}([\gamma_i,u_i]) \text{ and } \mu_{CZ}([\gamma_i, u_i]) =  2n.$ Such $[\gamma_i, u_i]$ exists by \emph{spectrality} of the invariant $c$ and non-degeneracy of $H_i$.  
  
By the Arzela--Ascoli theorem, a subsequence of the orbits $\gamma_i$, which we will denote by $\gamma_i$ as well, $C^1$--converges to an orbit $\gamma'$ of $H$.  Since the orbits $\gamma_i$ $C^1$--converge to $\gamma'$, one can construct a capping disc $u'$ for $\gamma'$ such that $\omega(u_i)$ converges to $\omega(u')$.  It follows that $c(H) = \mathcal{A}_H([\gamma',u'])$ and thus $\gamma' \in A$.  
  
Now, $\gamma'$ is isolated and $H_i$ coincides with $H$ on $U_{\gamma'}$.  Hence, $\gamma_i = \gamma'$ for large $i$ and thus $ \omega(u_i) = \omega( u')$ for large $i$.  It then follows that the capped orbits $[\gamma', u_i]$, for sufficiently large $i$, satisfy the conclusion of our lemma.  
  \end{proof}

\section{Proofs of the energy-capacity-type inequalities} \label{section theo base}
The main goal of this section is to prove Theorems \ref{theo base} and \ref{theo base rational}.  We will also state and prove two additional results which will be used in the proof of Theorem \ref{theo unicite complete}.
  \noindent Our arguments will use the following notion:
   \begin{definition}(See \cite[Definition 4.3]{usher})
   Let $f: M \rightarrow \mathbb{R}$ be an autonomous Hamiltonian.  A critical point $p$ of $f$ is said to be flat if the linearized flow $(\phi^t_f)_* : T_pM \rightarrow T_pM$ has no non-constant periodic orbits of period at most $1$. The function $f$ is called flat if all of its critical points are flat.
  \end{definition}
  The importance of the above notion stems from the fact that if $p$ is a non-degenerate and flat critical point of $f$, then the Morse index of $p$ coincides with the Conley--Zehnder index of $[p, u_p]$.  In Theorem 4.5 of \cite{usher}, Usher proves that a slow and autonomous Hamiltonian on a closed manifold can be $C^0$--approximated, up to any precision, by Hamiltonians which are slow, flat, and Morse. In our proof of Theorem \ref{theo base} we will need the following variant of Usher's theorem.
  \begin{theo} (Usher \cite[Theorem 4.5]{usher}) \label{usher's theorem}
  Let $H: M \rightarrow \mathbb{R}$ denote a slow Hamiltonian whose support is contained in $U$.  For any $\delta >0$ there exists a slow Hamiltonian $K:M \rightarrow \mathbb{R}$ such that $\Vert K - H \Vert_{C^0} < \delta$, the support of $K$ is contained in $U$, and  all critical points of $K$ that are contained in the interior of its support are non-degenerate and flat.
  \end{theo} 
  We will not prove the above theorem as it can easily be extracted from the proof of Theorem 4.5 in \cite{usher}.  Theorems \ref{theo base} and \ref{theo base rational} are similar in nature and their proofs have significant overlaps.  Hence, we will provide a single argument proving both theorems at once. 
\begin{proof}[Proofs of Theorems \ref{theo base} and \ref{theo base rational}]
Observe that by the shift property of spectral invariants we may assume, without loss of generality, that $C=0$.  
 
 For any $\delta > 0$ pick a time independent Hamiltonian $f\in C^{\infty}_c(U)$ such that $f$ is slow, $0 \leq f$, and $\chz(U) - \delta \leq \max(f)$.  Since $f$ is slow we have $c(f) = \max(f)$ and $c(-f) = 0$; for a proof of this fact see Proposition 4.1 of \cite{usher}.  Note that the conventions used in \cite{usher} are different from ours. By Theorem \ref{usher's theorem}, we may assume that the critical points of $f$ that are contained in the interior of the support of $f$ are non-degenerate and flat.  
  Consider the Hamiltonian $H_s = H+sf$.  Its 1--periodic orbits consist of 1--periodic orbits of the flow of $H$ together with the critical points of $f$.  Hence,
  $$ \Spec(H_s)= \Spec(H) \cup \{ s f(p) - \omega(\Sigma) : p \in \Crit(f), \; \Sigma \in \pi_2(M) \},$$ where $\Crit(f)$ denotes the set of critical points of $f$.  Similarly, define $\bar{H}_s = \bar{H} + sf$.  We have:
 $$ \Spec(\bar{H}_s)= \Spec(\bar{H}) \cup \{  s f(p) - \omega(\Sigma): p \in \Crit(f),  \; \Sigma \in \pi_2(M)  \}.$$
  By the spectrality property we know that $c(H_s) \in \Spec(H_s)$ and $c(\bar{H}_s) \in \Spec(\bar{H}_s)$.  However, suppose that one of the following two situations holds:
  \begin{align} 
&  c(H_s) \in \Spec(H) \text{ for all }s\in [0,1],   \label{super spectrality 1} \\
 \mbox{or}\quad &  c(\bar{H}_s) \in \Spec(\bar{H})\text{ for all }s\in [0,1].   \label{super spectrality 2}
  \end{align}
 If (\ref{super spectrality 1}) holds, then it follows, from the continuity property of spectral invariants, that $c(H) = c(H_1) = c(H+f)$.  Using the triangle inequality we obtain $c(f) \leq c(\bar{H}) + c(H + f).$ Combining these with the fact that $c(f) = \max(f)$, we get $$ \chz(U) - \delta \leq  \max(f) \leq c(\bar{H}) + c(H).$$  We arrive at the same conclusion if (\ref{super spectrality 2}) holds.
 
 We will next show that if the first possibility, in either of Theorems \ref{theo base} and \ref{theo base rational}, does not hold, then the second one must hold.   Therefore, \textbf{for the rest of the proof}, we will suppose that $\chz(U) > c(\bar{H}) + c(H)$.  This implies that there exist $\delta$ and $f$ as in the first paragraph of this proof such that (\ref{super spectrality 1}) and (\ref{super spectrality 2}) do not hold. Let $s_0 = \inf \{ s \in [0,1]: c(H_s) \notin \Spec(H) \}.$  Note that this means $c(H) = c(H_{s_0})$. Pick a sequence of numbers $s_i \in \{ s \in [0,1]: c(H_s) \notin \Spec(H) \}$ such that $s_i \to s_0$. There exist critical points $p_i$ of $f$ contained in the interior of the support of $f$ such that $c(H_{s_i}) = s_i f(p_i) - \omega(\Sigma_i)$, where $\Sigma_i \in \pi_2(M)$.  By passing to a subsequence, we may assume that $p_i \to p,$ where $p$ is a critical point of $f$; note that $p$ is not necessarily contained in the interior of the support of $f$. Now, the sequence $\omega(\Sigma_i)$ must converge because both  $c(H_{s_i})$ and $s_i f(p_i)$ converge.  Since $\omega(\pi_2(M))$ is discrete we conclude that $\omega(\Sigma_i) = \omega(\Sigma)$ for large $i$.  It then follows that $c(H_{s_i}) = s_i f(p_i) -\omega(\Sigma)$ for large $i$, and 
 \begin{equation}\label{eq: c(H)}
  c(H) = c(H_{s_0})= s_0 f(p) - \omega(\Sigma) .
 \end{equation}
 Similarly, let $r_0 = \inf \{ r \in [0,1]: c(\bar{H}_s) \notin \Spec(\bar{H}_s) \}$. 
  Repeating the same argument as above we find a capped orbit $[q, \Sigma']$ such that   
  \begin{equation}\label{eq: c(bar_H)}
  c(\bar{H}) = c(\bar{H}_{r_0}) = r_0 f(q) - \omega(\Sigma').
  \end{equation}
  We will prove Theorems \ref{theo base} and \ref{theo base rational} by carefully analyzing the numbers $\omega(\Sigma)$ and $\omega(\Sigma').$
  
\noindent \textbf{Proof of Theorem \ref{theo base rational}:}
 Since $(M, \omega)$ is rational, there exist integers $k_1$ and $k_2$ such that $\omega(\Sigma) = k_1 \Omega$ and $\omega(\Sigma') = k_2 \Omega$.  From Equations (\ref{eq: c(H)}) and (\ref{eq: c(bar_H)}) we get that $c(H)= s_0 f(p) - k_1 \Omega$ and $c(\bar{H}) =  r_0 f(q) - k_2 \Omega$ so that point (2) of Theorem \ref{theo base rational} holds with $k=-k_1$ and $\bar{k}=k_2$ since $0 \leq s_0 f(p), \; r_0 f(q) \leq \chz(U)$. 

Moreover, we have the following chain of inequalities:
 $$0 \leq c(H) + c(\bar{H}) = s_0 f(p) + r_0 f(q) + (k -\bar{k}) \Omega \leq \chz(U)$$
which implies, if $\chz(U) < \frac{1}{2} \Omega$, that $$-\Omega < -2 c_\mathrm{HZ}(U)\leq-s_0 f(p) - r_0 f(q)\leq (k -\bar{k}) \Omega \leq \chz(U)<\frac12\Omega$$
which can be satisfied only if $k=\bar{k}$.
 
\noindent \textbf{Proof of Theorem \ref{theo base}:} We now assume $(M,\omega)$ to be monotone.
We can apply Lemma \ref{CZ index lemma} to $H_{s_i}$ and assume that $\mu_\mathrm{CZ}([p_i, \Sigma_i]) = 2n.$  On the other hand, $H_{s_i}$ coincides with $s_i f$ on a neighborhood of $p_i$ and thus $$\mu_\mathrm{CZ}([p_i, \Sigma_i]) = i_{\text{Morse}}(p_i) - 2 c_1(\Sigma_i),$$
  where $i_{\text{Morse}}(p_i)$ is the Morse index of $p_i$ with respect to $f$.  Here, we have used the assumption that the critical points $p_i$ of $f$ are flat and non-degenerate, and hence $i_{\text{Morse}}(p_i) = \mu_\mathrm{CZ}([p_i, u_{p_i}])$.  Because $i_{\text{Morse}}(p_i) \leq 2n $ we conclude that $c_1(\Sigma_i)\leq 0.$  Recall that for large $i$, $\omega(\Sigma_i) = \omega(\Sigma)$, and thus, by monotonicity,  $c_1(\Sigma_i) =c_1(\Sigma).$  Therefore,  $$c_1(\Sigma)\leq 0.$$
Similarly, we have $$c_1(\Sigma')\leq 0.$$

  Recall that $\omega = \lambda c_1$ on $\pi_2(M)$.  First, suppose that $ \lambda > 0$.  Because $c_1(\Sigma), c_1(\Sigma')\leq 0$ we get that $$c(H) =  s_0 f(p) -\lambda c_1(\Sigma) \geq  s_0 f(p) \text{ and, }$$   $$ c(\bar{H}) =  r_0 f(q) - \lambda c_1(\Sigma') \geq r_0f(q).$$  Combining these inequalities with the assumption that $c(H)  + c(\bar{H}) < \chz(U)$ we conclude that 
  $$0 \leq c(H)  \leq  \chz(U) \; \text{ and } \; 0 \leq c(\bar{H})  \leq  \chz(U).$$  
 This proves Theorem \ref{theo base} for positively monotone symplectic manifolds.  Next, suppose that $\lambda \leq 0$ and repeat the same argument as in the previous paragraph to get that $c(H) \leq   s_0 f(p)$ and $c(\bar{H}) \leq r_0 f(q)$.  Thus, $c(H), \; c(\bar{H})  \leq \chz(U)$.  Combining this with the fact that $c(H)  + c(\bar{H}) \geq 0 $ we obtain
 $$|c(H)|  \leq \chz(U) \; \text{ and } \;| c(\bar{H}) | \leq \chz(U)$$
which concludes the proof of Theorem \ref{theo base}.
\end{proof}

We now focus on the rational case (proofs in the particular case of monotone manifolds are quite similar only slightly easier). Theorem \ref{theo base rational} has the following straightforward corollary.

\begin{corol}\label{corol rational 2 open sets} Let $U_-$ and $U_+$ denote non-empty open subsets of $(M,\omega)$, and $C_-$ and $C_+$  real numbers such that $\frac14\Omega>C_\pm > \chz(U_\pm)$. If a Hamiltonian $H$ satisfies $H|_{U_\pm}=\pm C_\pm$,  then $\gamma(\tilde{\phi}_H^{t})$ is greater than or equal to at least one of $\chz(U_-)$ and $\chz(U_+)$. 
 \end{corol}

 \begin{proof}
 Let $H$ be as above. Apply Theorem \ref{theo base rational} to $H$ on both $U_-$ and $U_+$ and get two integers $k$, $l$ such that
  \begin{align*}
    \gamma(\tilde{\phi}_H^t) \geq \chz(U_+) \quad &\mbox{or} \quad \left\{ 
      \begin{array}{l}
        C_+ +  k \Omega \leq c(H) \leq C_+ + k\Omega+ \chz(U_+) \\
        -C_+ - k\Omega \leq c(\bar{H}) \leq -C_+ -k\Omega+ \chz(U_+)
      \end{array}
\right.\\
    \gamma(\tilde{\phi}_H^t) \geq \chz(U_-) \quad &\mbox{or} \quad \left\{ 
      \begin{array}{l}
        -C_-+l\Omega \leq c(H) \leq -C_-+l\Omega+ \chz(U_-) \\
         C_--l\Omega \leq c(\bar{H}) \leq C_--l\Omega+ \chz(U_-)
      \end{array}
\right.
  \end{align*}
  
  Thus, either we directly get $\gamma(\tilde{\phi}_H^t) \geq \chz(U_\bullet)$ for $\bullet$ being either $+$ or $-$, or we have:
\begin{align*}
 &  c(H) + c(\bar{H}) \geq - C_-   - C_+  + (l-k)\Omega,\qquad \mbox{and} \\ 
 & c(H) + c(\bar{H})\leq - C_-   - C_+  + (l-k)\Omega + \chz(U_-) + \chz(U_+). 
\end{align*}
Since $0 \leq c(H) + c(\bar{H})$ the second inequality forces $l-k$ to be positive. Then from the first inequality we obtain: 
\begin{align*}
\gamma(\tilde{\phi}_H^t) = c(H) + c(\bar{H}) \geq\frac12 \Omega  \geq \chz(U_-) + \chz(U_+)
\end{align*}
which concludes the proof.
 \end{proof}

Now, using cut-off functions and this corollary, we can prove the following lemma which will be the main ingredient of the proof of Theorem \ref{theo unicite complete}.

\begin{lemma}\label{lemma for uniqueness theo}
  Let $F$ and $G$ be Hamiltonians. Let $U_\pm$ be non-empty, disjoint, open subsets such that $\chz(U_-) = \chz(U_+)$ (we denote this common value by $\chz(U)$) and 
  \begin{enumerate}
    \item $\chz(U) < \inf_{U_+} (F)-\sup_{U_+} (G)<\frac14\Omega$, and symmetrically\\ $-\frac14\Omega<\sup_{U_-} (F)-\inf_{U_-}(G) < -\chz(U)$, 
    \item $\osc_{U_\pm} (F) + \osc_{U_\pm} (G) < \frac{1}{3}  \chz(U)$.
 \end{enumerate}
 Then $\gamma(\tilde{\phi}_F^t,\tilde{\phi}_G^t) \geq \frac{1}{3} \chz(U)$.
\end{lemma}

\begin{proof}
  Fix $\varepsilon >0$. We choose disjoint open subsets $V_\pm$ such that $\overline{U_\pm}\subset V_\pm$ and $\osc_{V_\pm}(F) < \osc_{U_\pm}(F) + \varepsilon$ and $\osc_{V_\pm}(G) < \osc_{U_\pm}(G) + \varepsilon$. We also choose cut-off functions $\rho_\pm$ with support in $V_\pm$, such that $0\leq \rho_\pm\leq 1$ and $\rho_\pm|_{U_\pm}=1$.

We define intermediate functions, $f$ and $g$, by
\begin{align*}
  f = F - \rho_+(F-a_+) - \rho_-(F-a_-) \quad \mbox{with } a_+=\inf_{U_+}(F)  \mbox{ and } a_-=\sup_{U_-}(F),\\
  g = G - \rho_+(G-b_+) - \rho_-(G-b_-) \quad \mbox{with } b_+=\sup_{U_+}(G)  \mbox{ and } b_-=\inf_{U_-}(G).
\end{align*}
By triangle inequality, we get
\begin{align}
  \label{eq:TriangleIneqCoro8Compact}
  \gamma(\tilde{\phi}_F^t,\tilde{\phi}_G^t) \geq   \gamma(\tilde{\phi}_f^t,\tilde{\phi}_g^t) -  \gamma(\tilde{\phi}_f^t,\tilde{\phi}_F^t) -  \gamma(\tilde{\phi}_G^t,\tilde{\phi}_g^t)
\end{align}
and we now bound the quantities appearing on the right-hand side.

\noindent\textbf{Bounding $\gamma(\tilde{\phi}_f^t,\tilde{\phi}_g^t)$.} 
Define $\varphi =\bar{g}\#f$, that is,
\begin{align*}
  \varphi(t,x) = -g(t,{\phi}_g^t(x)) + f(t,{\phi}_g^t(x))
\end{align*}
which generates $({\phi}_g^t)^{-1}\circ {\phi}_f^1$. Notice that $\varphi$ is constant on both open sets $U_\pm$: $\varphi|_{U_\pm} = a_\pm - b_\pm$ and that, by assumption, 
\begin{align*}
  \frac14\Omega>C_+ &= a_+ - b_+ =\inf_{U_+}(F)-\sup_{U_+}(G) > \chz(U),  \\
  \frac14\Omega>C_- &= -(a_- - b_-) =-(\sup_{U_-}(F)-\inf_{U_-}(G)) > \chz(U)  .
\end{align*}
Thus, by applying Corollary \ref{corol rational 2 open sets} to $\varphi$ we get: $\gamma(\tilde{\phi}_f^t,\tilde{\phi}_g^t)=\gamma(\tilde{\phi}^t_\varphi)\geq \chz(U)$.

\noindent\textbf{Bounding $\gamma(\tilde{\phi}_f^t,\tilde{\phi}_F^t)$ and $\gamma(\tilde{\phi}_G^t,\tilde{\phi}_g^t)$.} By general property of $\gamma$, and definition of $f$
\begin{align*}
  \gamma(\tilde{\phi}_f^t,\tilde{\phi}_F^t) &\leq \osc_{M}(F-f) = \osc_M (\rho_+(F-a_+) + \rho_-(F-a_-)) \\
&\leq \osc_{V_+} \big(\rho_+(F-a_+)\big) + \osc_{V_-} \big(\rho_-(F-a_-)\big)  \\
&\leq \osc_{V_+} (F) + \osc_{V_-} (F) \leq \osc_{U_+} (F) + \osc_{U_-} (F) + 2\varepsilon
\end{align*}
For the same reasons, we also have $\gamma(\tilde{\phi}_g^t,\tilde{\phi}_G^t) \leq \osc_{U_+} (G) + \osc_{U_-} (G) + 2\varepsilon$ and \eqref{eq:TriangleIneqCoro8Compact} leads to 
\begin{align*}
  \gamma(\tilde{\phi}_F^t,\tilde{\phi}_G^t) &\geq  \chz(U) - ( \osc_{U_+} \!(F)+ \osc_{U_+} \!(G)) - (\osc_{U_-} \!(F)+\osc_{U_-} \!(G)) - 4 \varepsilon \\
&\geq \frac{1}{3}  \chz(U) - 4 \varepsilon
\end{align*}
for any $\varepsilon >0$. This concludes the proof.
\end{proof}

\section{Uniqueness of generators}\label{section preuve unicite} 

In this section, we prove Theorems \ref{theo unicite complete} and \ref{theo unicite hameo}.

\begin{proof}[Proof of Theorem \ref{theo unicite complete}]
  Assume that the conclusion of the theorem is false; i.e. $H-H'$ is a function of time and space variables on $I \times U$. Then there exist $t_0$ and, up to a shift of (say) $H'$ by a constant, $x_+\neq x_- \in U$ such that 
$$\Delta=H(t_0,x_+)-H'(t_0,x_+) = H'(t_0,x_-) - H(t_0,x_-) > 0.$$

First, notice that there exist $\delta_0\in]0,1]$ and $r_0>0$ such that for any $\delta\leq \delta_0$, $J=[t_0,t_0+\delta]\subset I$ and for any $r\leq r_0$, the balls $B_\pm=B_{r}(x_\pm)$ are disjoint, included in $U$, and
  \begin{align}\label{eq:IntervalForSigma}
    \left( \frac{5}{4} \frac{1}{M_+}, \frac{4}{3} \frac{1}{M_+} \right) \cap \left( \frac{5}{4} \frac{1}{M_-}, \frac{4}{3} \frac{1}{M_-} \right) \neq \emptyset
  \end{align}
with $M_+= \sup_{J\times B_+} (H)-\inf_{J\times B_+} (H')$ and $M_-= \sup_{J\times B_-} (H')-\inf_{J\times B_-} (H)$. (Even though $J$, $B_\pm$, and $M_\pm$ depend on $\delta$ and/or $r$, we omit them from the notation for readability.) Indeed, let $\eta=\frac{\Delta}{32}$ and choose $\delta_0$ and $r_0$ small enough such that
\begin{align*}
  \left\{ \begin{array}{ll}
    \sup_{J\times B_+}(H) \leq H(t_0,x_+) + \eta \quad &\mbox{and} \quad   \inf_{J\times B_+}(H') \geq H'(t_0,x_+) - \eta \\
    \inf_{J\times B_-}(H) \geq H(t_0,x_-) - \eta \quad &\mbox{and} \quad   \sup_{J\times B_-}(H') \leq H'(t_0,x_-) + \eta  
  \end{array}  \right.
\end{align*}
Then $\Delta \leq M_\pm \leq \Delta + 2 \eta$, so that $|M_+-M_-|\leq 2\eta = \frac{\Delta}{16}$ which in turn ensures that \eqref{eq:IntervalForSigma} holds. Next, notice that we can also assume $\delta_0$ and $r_0$ are small enough so that, for any $\delta\leq \delta_0$ and $r\leq r_0$ 
\begin{align} \label{eq:technical lower bounds}
\inf_{J\times B_+} (H)-\sup_{J\times B_+} (H')> \frac{4}{5} M_+    \mbox{ and }   \inf_{J\times B_-} (H')-\sup_{J\times B_-} (H)> \frac{4}{5} M_- 
\end{align}
(since these inequalities obviously hold for $\delta=0$ and $r=0$ and $H$ and $H'$ are continuous). We choose such a $\delta$. 

Recall that $\chz(B_+)=\chz(B_-)=\pi r^2$ (which we denote $\chz(B)$) so that we can choose $r$ small enough such that $\chz(B)<\delta \frac{3}{4} \Delta$. This in particular implies that $\chz(B)<\delta \frac{3}{4} M_\pm$. Finally, we choose $r$ small enough so that
\begin{align} \label{eq:rational condition}
\chz(B) < \frac{3}{16}\Omega \;.
\end{align}
Now that $r$ and $\delta$ are fixed, we choose $\sigma$ such that 
  \begin{align*}
   \frac{\delta\sigma}{\chz(B)} \in \left( \frac{5}{4} \frac{1}{M_+}, \frac{4}{3} \frac{1}{M_+} \right) \cap \left( \frac{5}{4} \frac{1}{M_-}, \frac{4}{3} \frac{1}{M_-} \right).
  \end{align*}
Notice that, by definition, $\sigma<\frac{4}{3} \frac{1}{M_+}\frac{\chz(B)}{\delta} \leq1$. This implies that for all $t\in[0,1]$, $t_0+\sigma\delta t\in J$ and we define $L_k$, $L'_k$, $L$ and $L'$ by: $L_\bullet^\star (t,x)=\sigma\delta H_\bullet^\star (t_0+\sigma\delta t,x)$ (with $\bullet$ being either nothing or an integer and $\star$ being either nothing or $'$).  

In view of the constants we chose, we get that 
\begin{align*}
  \inf_{[0,1]\times B_+}(L) &\geq \delta\sigma \inf_{J\times B_+}(H) \geq  \delta\sigma\left(\frac{4}{5} M_++ \sup_{J\times B_+} (H')\right) & \mbox{by \eqref{eq:technical lower bounds}}\\
&\geq   \delta\sigma \frac{4}{5} M_+ +\sup_{[0,1]\times B_+} (L')
\end{align*}
so that, by definition of $\sigma$,
\begin{align}\label{eq:infLminussupL'}
  \inf_{[0,1]\times B_+}(L)-\sup_{[0,1]\times B_+} (L')  \geq \delta\sigma \frac{4}{5} M_+  > \chz(B)
\end{align}
We also get:
\begin{align*}
  \sup_{[0,1]\times B_+}(L) &\leq \delta\sigma \sup_{J\times B_+}(H) =  \delta\sigma\left(M_++ \inf_{J\times B_+} (H')\right) \\
&\leq   \delta\sigma M_+ +\inf_{[0,1]\times B_+} (L')
\end{align*}
so that 
\begin{align*}
 \sup_{[0,1]\times B_+}(L)-\inf_{[0,1]\times B_+} (L')<\frac{4}{3} \chz(B) 
\end{align*}
which (together with \eqref{eq:infLminussupL'}) leads to
\begin{align*}
  \osc_{[0,1]\times B_+}(L)+\osc_{[0,1]\times B_+} (L') < \frac{1}{3} \chz(B) \;.
\end{align*}
Since the quantity $\osc_{[0,1]\times B_+}(L)+\osc_{[0,1]\times B_+} (L')$ is non-negative, the bound on the capacity of $B$ in terms of the rationality constant $\Omega$,  \eqref{eq:rational condition}, finally ensures that:
\begin{align*}
  \inf_{[0,1]\times B_+}(L)-\sup_{[0,1]\times B_+} (L')  \leq \sup_{[0,1]\times B_+} (L)  - \inf_{[0,1]\times B_+}(L')<\frac{4}{3} \chz(B) < \frac14 \Omega \;.
\end{align*}
By collecting all the above results, we get
\begin{align*}
& \chz(B)  < \inf_{[0,1]\times B+}(L)-\sup_{[0,1]\times B_+} (L')< \frac14 \Omega\;,  \quad \mbox{and} \\
& \osc_{[0,1]\times B_+}(L)+\osc_{[0,1]\times B_+} (L') < \frac{1}{3} \chz(B) \;
\end{align*}
and since $L_k$ and $L'_k$ converge uniformly on $U$ to $L$ and $L'$ respectively, they also satisfy all these inequalities as soon as $k$ is large enough.

Now, by considering the situation on $B_-$, we obtain the (symmetric) properties required in order to apply Lemma \ref{lemma for uniqueness theo} which allows us to conclude that $\gamma(\tilde{\phi}^t_{L_k},\tilde{\phi}^t_{L'_k})$ is bounded from below by $\frac{1}{3} \chz(B)$ for $k$ big enough.

However, by the definition of $L_k$, for all $t\in [0,1]$, $\phi_{L_k}^t=\phi_{H_k}^{t_0+\delta\sigma t}(\phi_{H_k}^{t_0})^{-1}$ so that $\tilde\phi_{L_k}^t=\tilde\phi_{H_k}^{t_0+\delta\sigma t} (\tilde\phi_{H_k}^{t_0})^{-1}$. Similarly, we have $\tilde\phi_{L'_k}^t=\tilde\phi_{H'_k}^{t_0+\delta\sigma t} (\tilde\phi_{H'_k}^{t_0})^{-1}$ and thus assumption \textit{(i)} ensures that $\gamma(\tilde{\phi}^t_{L_k},\tilde{\phi}^t_{L'_k})$ does go to 0 when $k$ goes to infinity and we get a contradiction. 
\end{proof}

As promised in the introduction, we will explain how one can recover Theorem \ref{theo unicite hameo} from Theorem \ref{theo unicite complete}.  But, before doing so, we make a short digression to discuss an important property of $C^0$--convergence.
 
\begin{remark}\label{rem: C0_dist}
An important feature of $C^0$--convergence, which will be used below, is that if a sequence of homeomorphisms, $\phi_i,$ $C^0$--converges to a \emph{homeomorphism} $\phi$, then the sequence of inverses, $\phi_i^{-1},$ $C^0$--converges to $\phi^{-1}$.  We will sketch a proof of this fact below.

First, note that $d_{C^0}(\phi_i, \phi) \to 0$ implies that $d_{C^0}(\psi \phi_i, \psi \phi) \to 0$ for any uniformly continuous map $\psi$.  
Taking $\psi = \phi^{-1}$, we get that $d_{C^0}(\phi^{-1} \phi_i, Id) \to 0.$  Next, observe that $d_{C^0}$ is right-invariant and so we get that $$d_{C^0}( \phi^{-1} \phi_i , Id) = d_{C^0}(\phi^{-1}, \phi_i^{-1}) \to 0.$$

The above proof would fail without the assumption that $\phi^{-1}$ exists.  In fact, it is possible for a sequence of homeomorphisms to converge (with respect to the above version of $d_{C^0}$), to a map which is not a homeomorphism.  In that case, the sequence of inverses diverges.  
Some authors use a version of $d_{C^0}$ which avoids the above issue by making it impossible for a sequence of homeomorphisms to converge to a map which is not a homeomorphism.   For example, this is achieved in \cite{muller-oh} by defining  $$\bar{d}_{C^0}(\phi,\psi) = \max_{x} (d(\phi(x),\psi(x)) + d(\phi^{-1}(x),\psi^{-1}(x))).$$  As pointed out by M\"uller and Oh, the group of homeomorphisms equipped with $\bar{d}_{C^0}$ is a complete metric space.
\end{remark}

\begin{proof}First, note that Theorem \ref{theo unicite hameo} follows from the following simpler statement: If $H_k$ is a sequence of normalized smooth Hamiltonians which uniformly converges to some continuous function $H$ and if the flows $\phi_{H_k}^t$ converge uniformly to $\Id$, then $H=0$. 

Then remark that if we knew that the spectral distance is continuous with respect to the $C^0$--topology then this statement would follow directly from Theorem \ref{theo unicite complete} on rational symplectic manifolds. Unfortunately, this is only partially known and we need a trick to get around this difficulty.  We are going to show that for any connected and sufficiently small open subset $U\subset M$, the function $H$ only depends on the time variable $t$. Since $H$ is normalized, this will prove the statement. We use the same trick as in \cite[Theorem 11]{buhovsky-seyfaddini}.

Let $U$ be an open connected subset of $M$ small enough to admit a symplectic embedding to a closed rational symplectic manifold $\iota:U\hookrightarrow W$. Let $\psi$ be a Hamiltonian diffeomorphism generated by a Hamiltonian function compactly supported in $U$. Since $\phi_{H_k}^t$ $C^0$--converges to $\Id$, the isotopy $\phi_{H_k}^{-t}\psi^{-1}\phi_{H_k}^t\psi$ is supported in $U$ for $k$ large enough. Moreover, by Remark \ref{rem: C0_dist}, it converges to $\Id$ in the $C^0$ sense. We may pushforward this isotopy using the embedding $\iota$ and get a Hamiltonian isotopy of $W$ supported $\iota(U)$. This isotopy also converges to $\Id$ in the $C^0$ sense. Thus, according to \cite[Theorem 1]{seyfaddini11}, its spectral pseudo-norm converges to 0. In other words, $\gamma(\tilde{\psi}^{-1}\tilde{\phi}_{H_k}^t\tilde{\psi},\tilde{\phi}_{H_k}^t)$ converges to 0. We may now apply our Theorem \ref{theo unicite complete}  in $W$ and get that $H(t,\psi(x))-H(t,x)$ only depends on the time variable on $[0,1]\times U$. Since this holds for any $\psi$, this proves our claim that $H$ depends only on the time variable on $[0,1]\times U$.
\end{proof}

\section{$C^0$-rigidity of the Poisson bracket}

This section is devoted to the proof of Theorem \ref{theo Poisson}

\begin{proof} We use the notation of Theorem \ref{theo Poisson}.  The assumption $G\in C^0_{\Ham}$ means that there exists a sequence of smooth functions $G_k'$ (a priori different from $G_k$), which converges uniformly to $G$ and such that the flows $\phi_{G_k'}^t$ converge in the $C^0$ sense to a continuous isotopy also denoted $\phi_G^t$. Let $s$ be a real number. We want to prove that $F=F\circ\phi_G^s$. The sequence of functions $F_k'=F_k\circ\phi_{G_k'}^s$ converges uniformly to $F\circ\phi_G^s$. In view of Theorem \ref{theo unicite complete}, if we show that $\gamma(\tilde{\phi}^t_{F_k},\tilde{\phi}^t_{F_k'})$ converges to 0, then $F=F\circ\phi_G^s$ follows.

Let us recall two identities. For any smooth functions $H$, $K$,
\begin{equation}\label{equation 1} \phi^t_{H\circ\phi_K^s}=\phi^{-s}_{K}\circ\phi^t_{H}\circ\phi^{s}_{K},
\end{equation}
\begin{equation}\label{equation 2} H\circ\phi_K^s-H=\int_0^s\{H,K\}\circ\phi_K^\sigma\,d\sigma.
\end{equation}
The triangle inequality for $\gamma$ and (\ref{equation 1}) give
$$\gamma(\tilde{\phi}^t_{F_k},\tilde{\phi}^t_{F_k'})\leq \gamma(\tilde{\phi}^t_{F_k\circ\phi_{G_k}^s},\tilde{\phi}^t_{F_k})+\gamma(\tilde{\phi}^{-s}_{G_k}\tilde{\phi}^t_{F_k}\tilde{\phi}^{s}_{G_k},\tilde{\phi}^{-s}_{G_k'}\tilde{\phi}^t_{F_k}\tilde{\phi}^{s}_{G_k'}).$$
The Lipschitz properties of $\gamma$ with respect to the $C^0$--norm of Hamiltonians, the bi-invariance of $\gamma$ and (\ref{equation 2}) yield
$$\gamma(\tilde{\phi}^t_{F_k},\tilde{\phi}^t_{F_k'})\leq
 s\|\{F_k,G_k\}\|_{C^0}+2\|G_k-G_k'\|_{C^0}.$$
Hence $\gamma(\tilde{\phi}^t_{F_k},\tilde{\phi}^t_{F_k'})$ converges to 0 as wanted.
\end{proof}

\nocite{*}
\bibliographystyle{plain}
\bibliography{biblio}

\end{document}